\documentclass[a4paper, 11pt]{amsart}
\usepackage{amsmath,amssymb,amsthm}
\usepackage{MnSymbol}

\theoremstyle{plain}
\newtheorem{thm}{Theorem}[section]

\newtheorem{propn}[thm]{Proposition}

\theoremstyle{definition}
\newtheorem{defn}[thm]{Definition}
\newtheorem{xmpl}{Example}

\theoremstyle{remark}
\newtheorem*{rmk}{Remarks}

\newcommand{\pd}[2]{\frac{\partial#1}{\partial#2}}
\newcommand{\vp}[1]{\frac{\partial}{\partial#1}}

\newcommand{\ode}[1]{\frac{\text{d}}{\text{d}#1}}

\renewcommand{\d}{\text{d}}
\newcommand{\Sp}{\text{Sp}}
\newcommand{\R}{\mathbb{R}}

\def\hook{{\mathchoice{\vrule height 0pt depth 1pt width 5pt
                         \vrule height 8pt depth 1pt \kern 3pt}
                      {\vrule height 0pt depth 1pt width 5pt
                         \vrule height 8pt depth 1pt \kern 3pt}
                      {\vrule height 0pt depth 0.5pt width 3pt
                         \vrule height 6pt depth 0.5pt width 0.5pt \kern 1pt}
                      {\vrule height 0pt depth 0.5pt width 3pt
                         \vrule height 6pt depth 0.5pt width 0.5pt \kern 1pt}
                         }}

\usepackage{tikz}
\usetikzlibrary{shapes,snakes}
\usepackage{pgfplots}
\usetikzlibrary{3d}
\usetikzlibrary{calc}
 \usetikzlibrary{backgrounds}

 \usetikzlibrary{through}

\usetikzlibrary{decorations.text}
\usetikzlibrary{arrows}

\usepackage{color}

\usepackage{mathtools}

\usepackage{arydshln}

\usetikzlibrary{matrix,arrows,patterns,decorations.pathmorphing}

\begin{document}
\title[\resizebox{4.5in}{!}{Evolution equations: Frobenius integrability, conservation laws and travelling waves}]{Evolution equations: Frobenius integrability, conservation laws and travelling waves}
\author{Geoff Prince and Naghmana Tehseen}
\address{ Department of Mathematics and Statistics, La Trobe University, Victoria, 3086, Australia}
\email{g.prince@latrobe.edu.au, \ n.tehseen@latrobe.edu.au}
\date{\today}
\noindent \begin{abstract}
We give new results concerning the Frobenius integrability and solution of evolution equations admitting travelling wave solutions. In particular, we give a powerful result which explains the extraordinary integrability of some of these equations.  We also discuss ``local'' conservations laws for evolution equations in general and demonstrate all the results for the Korteweg de Vries equation.
 \end{abstract}
\maketitle

\section{Introduction}

\noindent Travelling wave solutions of nonlinear partial differential equations (PDEs) have been studied extensively. Our aim is to present a geometric perspective on these solutions, viewed as the projection to the configuration space of the common level sets of functions on a jet bundle. In particular, these common level sets are the integral manifolds of a certain Frobenius integrable distribution, the Vessiot distribution.

 Our recent work  \cite{NT14,NTGP13}, encodes a system of PDEs as an exterior differential system (EDS) and we have provided a systematic approach to solving second order linear and non-linear PDEs in the presence symmetries ( known as solvable structures), a generalisation of the usual reduction by symmetry process, and based on the Vessiot distribution. We present an application of this technique to find travelling wave solutions of evolution PDEs. In doing so we give a new result (theorem \ref{propn 1}) identifying necessary criteria for the sort of integrability enjoyed by the Korteweg de Vries soliton.

The layout of this paper is as follows: In section 2 we review our approach to the geometric treatment of PDEs. We define solvable structures and discuss their utility for Frobenius integrable distributions. In this context we give a new result which will be important in discussing travelling wave solutions of evolution equations. In section 3 we discuss the travelling wave solutions of general nonlinear evolution PDEs and we present new results about their integrability.  In section 4 we give a new presentation of the  Korteweg de Vries soliton  as an example. In section 5 we describe the relationship between conserved densities, constants of the motion and first integrals of the Vessiot distribution.

\section{Solvable Structures and PDEs}

\noindent Consider a system of partial differential equations  \cite{Stormark} of $m$ independent variables $x^i$ and $n$ dependent variables  $u^j$,
\begin{align}G^a(x^i,u^j,u^j_{i_1},u^j_{i{_1}i{_2}}\ldots u^j_{i_1\ldots i_k})=0. \label{eqn:g.pde}\end{align}
The subscripts $1\leq i_1\leq\ldots \leq i_k\leq n$ are used to specify the partial derivatives of $u^j$, where $k$ is the maximum order of the system. We denote the total space by $E=X\times U$ as the product of the spaces $X$ and $U$ of the independent and dependent variables respectively. Usually $X=\R^m$ and $U=\R^n.$

A system of partial differential equations can be regarded as a submanifold $S$ of the jet bundle $J^k(\R^m,\R^n),$ and so the associated contact Pfaffian system on $J^k$ induces a Pfaffian system on $S$ by restriction. In this way we can think of the PDE as a manifold with a Pfaffian system. The solutions of the PDE correspond to $m-$dimensional integral manifolds of the Pfaffian system on which $dx^1\wedge\ldots \wedge dx^m\neq 0.$ This inequality is usually called an independence conditions and simply requires that our integral submanifolds suitably project to the graph space $X \times U$.

The contact Pfaffian system $\Omega^k(X,U)$ consists of all one$-$forms whose pull-back by a prolonged section of $X\times U$ vanishes. Locally it is spanned by the contact forms
\[\theta^j_I:=d u^j_I-\sum_{i=1}^m u_{I, i}^j d x^i,\]
where $I$ is a multi-index of order less than or equal to $k-1.$

The dual distribution $\Omega^{k^\ast}(X,U)$ consists of all vector fields annihilated by contact forms $\Omega^k(X,U)$. A straightforward calculation \cite{Fesser09} shows that it is generated by
\begin{align}\label{Vessiot dist}
V_{i}^{(k)}&:=\partial_i+\sum_{j=1}^n \sum_{0\leq \left|I\right|<k}u_{I,i}^j\partial_{u^j_I},\quad 1\leq i\leq m,\nonumber \\
V_{j}^{I}&:=\partial_{u^j_I},\quad \quad \left|I\right|=k,\quad \quad 1\leq j \leq n.
\end{align}
The Vessiot distribution \cite{Fackerell85,Vessiot1924} is essentially the restriction of this dual distribution to the submanifold defined by our differential equation condition $G = 0$.  In general, the Vessiot distribution is not Frobenius integrable.

Vessiot \cite{Vessiot1924} has given an algorithm for constructing all the Frobenius integrable sub-distributions of any given distribution. If the distribution is not Frobenius integrable, then it is interesting to ask for the sub-distributions which are Frobenius integrable. He looks for generic sub-distributions that satisfy the algebraic constraint called by him \emph{involutions of degree r} and then shows that maximal such involutions can be deformed so as to be Frobenius integrable. He does this via the Cauchy-Kowalevski theorem. We have given an equivalent method \cite{NT14,NTGP13} to find the largest integrable Vessiot sub-distributions ({\em reduced Vessiot distributions}). We showed how to use a solvable  structure (see below) to integrate a PDE in the original coordinates, but most importantly, we showed how to impose a solvable structure on a PDE so as to determine particular largest solvable sub-distributions of the Vessiot distribution. For more details see \cite{NT14,NTGP13}.

 So the first problem is to locate the largest integrable sub-distributions which satisfy the transverse condition. Then we apply the integrating factor technique \cite{sherr} to integrate these integrable Vessiot sub-distributions. The paper by Sherring and Prince \cite{sherr} extends Lie's approach to integrating a Frobenius integrable distribution by using solvable structures. First we need some basic definitions.

\begin{defn}
  A differential $p-$form $\Omega$ is {\em simple or decomposable} if it is the wedge product of $p$ $1-$forms.
\end{defn}
\begin{defn}
A {\em constraint} $1-$form $\theta$ for differential form $\Omega$ is any $1-$form satisfying $\theta\wedge{\Omega}=0,$ which implies $Y\righthalfcup\theta=0,~ \forall Y\in \text{ker}~{\Omega}.$
\end{defn}
\begin{defn}
A {\em characterising form} for an $p-$dimensional (vector) distribution $D$ is a form on $M^n$ of degree $(n-p)$ which is the exterior product of $(n-p)$ constraint forms.
\end{defn}
\begin{defn}
Let $\Omega\in\bigwedge^p(M^n)$ for some $p>1$ be decomposable, $\Omega$ is {\em Frobenius integrable} if $d\Omega=\lambda\wedge\Omega.$ Equivalently, $D:=\text{ker}\ \Omega$ is Frobenius integrable, that is, closed under the Lie bracket. Note that $\text{dim(ker}\ \Omega)= n-p,$ since $\Omega$ is simple.
\end{defn}

\noindent A {\em solvable structure} is then defined as follows:

 \begin{defn}
   Let $\Omega$ be a characterising $p-$form for an $(n-p)$ dimensional distribution $D$ on a manifold $M^n.$ An ordered set of $p$ linearly independent vector fields $\{X_1,\dots,X_p\}$ with $\text{dim}(\Sp\{X_1,\ldots,X_p\}\bigoplus D)=n$  forms a {\em solvable structure} \cite{sherr} for $\Omega$ (equivalently $D$) if the sequence of simple forms  $$\Omega,\ X_1 \righthalfcup \Omega,\ \dots\ ,\ X_{p-1}\righthalfcup \dots\righthalfcup  X_1\righthalfcup  \Omega$$ satisfies
 \begin{align*}
 \mathcal{L}_{X_1}\Omega &=\ell_1\Omega,\\
 \mathcal{L}_{X_2}(X_1\righthalfcup\Omega) &=\ell_2(X_1\righthalfcup\Omega),\\
 \vdots\\
 \mathcal{L}_{X_{p}}(X_{p-1}\righthalfcup\dots\righthalfcup X_1\righthalfcup\Omega) &=\ell_{p}(X_{p-1}\righthalfcup\dots\righthalfcup X_1\righthalfcup\Omega),
 \end{align*}
  for some smooth functions $\ell_1,\ldots,\ell_p$ on $M^n.$
\end{defn}
The Lie symmetry determination software package {\tt dimsym} \cite{sherrdimsym} can be used to locate solvable structures for distributions.

 Our main tool to integrate a Frobenius integrable distribution using solvable structures is the following theorem from \cite{sherr}:
\begin{thm} \label{sherr result}
 Let $\Omega$ be a decomposable $k-$form on a manifold $M^n,$ and let $\Sp(\{X_1,\ldots, X_k\})$ be a $k-$dimensional distribution on an open $U\subseteq M^n$ satisfying $X_i\righthalfcup \Omega\neq 0$ everywhere on $U.$ Further suppose that $\Sp(\{X_{j+1},\ldots, X_k\}\bigcup\text{ker}~\Omega)$ is integrable for some $j<k$ and that $X_i$ is a symmetry of $\Sp(\{X_{i+1},\ldots, X_{k}\}\bigcup \text{ker}~\Omega)~\text{for}~ i=1,\ldots, j.$

 Put $\sigma^i:=X_1\righthalfcup \ldots \righthalfcup \bar{X}_i\righthalfcup\ldots\righthalfcup X_k\righthalfcup\Omega,$
where $\bar{X}_i$ indicates that this argument is missing and $\omega^i:=\frac{\sigma^i}{X_i\righthalfcup \sigma^i}~ \text{for}~i=1,\ldots,k$
so that $\{\omega^1,\ldots, \omega^k\}$ is dual to $\{X_1,\ldots, X_k\}.$ Then $d \omega^1=0;~d \omega^2=0~\mod~\omega^1;~d \omega^3=0~\mod~\omega^1,\omega^2;~~~\ldots~;~ d \omega^j=0~\mod~\omega^1,\ldots,\omega^{j-1},$ so that locally
 \begin{align*}
   \omega^1&=d \gamma^1,\\
   \omega^2&=d \gamma^2-X_1(\gamma^2)d\gamma^1,\\
   \omega^3&=d \gamma^3-X_2(\gamma^3)d\gamma^2-(X_1(\gamma^3)
   -X_2(\gamma^3)X_1(\gamma^2))d\gamma^1,\\ \vdots\\
  \omega^{j}&\equiv d \gamma^{j}~\mod~d\gamma^1,\ldots, d\gamma^{j-1},
 \end{align*}
 for some $\gamma^1,\ldots,\gamma^j\in \bigwedge^0 T^* U.$ The system $\{\omega^{j+1},\ldots,\omega^k\}$ is integrable modulo $d\gamma^1,\ldots, d\gamma^j$ and locally $\Omega=\gamma^0 d\gamma^1\wedge d\gamma^2 \wedge \ldots \wedge d\gamma^j\wedge \omega^{j+1}\wedge\ldots \wedge \omega^k ~\text{for some}~\gamma^0\in \bigwedge^0(T^* U).$ Each $\gamma^i$ is uniquely defined up to the addition of arbitrary functions of $\gamma^1,\ldots,\gamma^{i-1}.$
 \end{thm}
We will now demonstrate the utility of a closed characterising $p$-form, $\Omega,$ on $M^n.$ Firstly note that for $p<n$ the closure of $\Omega$ implies its Frobenius integrability. As shown in the proposition below, in the presence of a solvable structure this closure allows the immediate identification of closed constraint one-forms (factors of $\Omega$) for the distribution in question. (So far as we know this is a new result.)

\begin{propn}\label{propn2}
Let  $\Omega$ be a closed characterising $p$-form  for a distribution $D$ on $M^n$ with $p<n$. Let  $\{X_1,\dots,X_p\}$ be a solvable structure for $\Omega$ with the additional properties that
  \begin{align*}
  \mathcal{L}_{X_p}\Omega=&\ 0,\\
  \mathcal{L}_{X_{p-1}}(X_p\righthalfcup\Omega)=&\ 0,\\
  \vdots &\ \nonumber\\
\mathcal{L}_{X_1}(X_2\righthalfcup\dots\righthalfcup X_p\righthalfcup \Omega)=&\ 0.
\end{align*}
Then the forms in the following sequence are all simple and closed and the last is a closed one-form factor of $\Omega$:
\begin{equation*}
\Omega,\ X_p \righthalfcup \Omega,\ \dots \ , X_2\righthalfcup\dots\righthalfcup X_p\righthalfcup \Omega.
\end{equation*}
Moreover, if $\Omega(X_p,\dots,X_1)$ is not constant (it is necessarily non-zero) then
\begin{equation*}
d(\Omega(X_p,\dots,X_1))\wedge(X_2\righthalfcup\dots\righthalfcup X_p\righthalfcup \Omega)=0,
\end{equation*}
effectively integrating the last closed form in the sequence.

\end{propn}

\begin{proof}
Note that simple forms are exactly those with maximal (and constant) dimension kernels. If $\omega$ is a simple form and $Z$ is not in its kernel then $Z \righthalfcup\omega$ is simple because its kernel is of dimension one greater than that of $\omega.$ So each form in the above sequence is simple. \newline
The closure of each form follows from the Lie derivative formula: $\mathcal{L}_X\alpha=X\righthalfcup d\alpha+d(X\righthalfcup\alpha)$ and the closure of $\Omega.$ For example, $d(X_p \righthalfcup\Omega)=\mathcal{L}_{X_p}\Omega - X_p\righthalfcup d\Omega=0.$
\end{proof}

\begin{rmk} In the best case, where every ordering of $1,\dots, p$ produces a solvable structure from $\{X_1,\dots,X_p\},$ this proposition produces all the closed one-form factors of $\Omega$ and we explicitly find the integral manifolds of the corresponding distribution.
\end{rmk}

In the following section, we will consider an evolution PDE of one dependent variable and two independent variables and explain the application of our technique developed in \cite{NT14,NTGP13} to find the travelling wave solutions.

\section{Evolution PDE and travelling wave solution}
Suppose we have an evolution PDE of order $k$ of one dependent variable $u$ and two independent variables $t,x$ given by
\begin{equation}\label{Gevol}
  u_{t}=F(t,x,u,u_x,u_{xx},u_{xxx},\ldots),
\end{equation}
for some smooth function $F.$ The embedded codimension one submanifold
\[S:=\{(t,x,u,u_t,u_x,u_{xx},u_{tx},\ldots)\in J^k(\R^2,\R)~ |~ u_{t}-F(t,x,u,u_x,u_{xx},\ldots)=0\},\]
is a subset of $J^k(\R^2,\R).$ The local solution of the PDE is described by the map
\[i:(t,x,u,u_x,u_{xx},u_{tx},\ldots)\hookrightarrow (t,x,u,F,u_x,u_{xx},u_{tx},\ldots).\]

We are interested in travelling wave solutions \cite{Drazbook89}, $u(x,t)=f(x-ct)$, of the evolution equation \eqref{Gevol}. Equivalently,
\begin{equation}\label{Tw:eq}
  u_t+cu_x=0.
\end{equation}
We add equation \eqref{Tw:eq} and its differential consequences, via the RIF algorithm \cite{Maple,Reid96}, as extra conditions to \eqref{Gevol}. These differential consequences are
\begin{align*}
u_{tx}=-cu_{xx},\quad u_{tt}=c^2u_{xx},\quad u_{ttt}=-c^3u_{xxx},\quad u_{ttx}=c^2u_{xxx},\quad u_{txx}=-cu_{xxx},\ldots
\end{align*}
The map $i$ becomes
\[i:(t,x,u,u_x,u_{xx},u_{xxx},\ldots)\hookrightarrow (t,x,u,F,u_x,u_{xx},-cu_{xx},c^2u_{xx}\ldots).\]

The $k^{th}$ order pulled-back (by $i^*$) contact system $D_V^{\ast}$ on $J^k(\R^2,\R)$ is generated by

\begin{align}
\theta^1&:=du+\frac{F}{c}(dx-cdt),\nonumber\\
  \theta^2&:=dF+cu_{xx}(dx-cdt),\\
  &\vdots \nonumber\\
  \theta^{|A|}&:=du_{A-1}-u_{A}(dx-cdt),\quad A=(x,\ldots,x),\quad 3\leq |A| \leq k.\nonumber
\end{align}
with standard characterising form $\Omega_\theta:=\theta^1\wedge\ldots\wedge\theta^k.$\newline
The next step is to determine the Frobenius integrability of this contact system, equivalently of $\Omega_\theta$, i.e.
\begin{align*}
  d\theta^a\wedge\Omega_\theta=0 \iff \theta^a\wedge\d\Omega_\theta=0, \quad a=1,\dots,k
\end{align*}

Once we have a Frobenius integrable Vessiot distribution its integration follows, for example, using solvable symmetry structures (see Theorem \ref{sherr result}). Projection of these integral submanifolds to $\R^2\times\R$ produces graphs of solutions of  our PDE.

The next result is key to deciding whether an evolution equation has a Frobenius integrable Vessiot distribution without any reduction. It explains why, for example, the Korteweg de Vries, is so amenable to solution as we will see.

\begin{thm}\label{propn 1}
  Let $\Omega_\theta$ be the standard characterising $k-$form for an evolution PDE \eqref{Gevol} on $J^k(\R^2,\R),~k\geq 3$ with the travelling wave ansatz \eqref{Tw:eq}.
  \begin{itemize}
  \item[1.] $\Omega_\theta$ is Frobenius integrable if and only if \newline
  $$dF\wedge\phi\equiv 0\ \text{mod}\ \{du,du_{xx},\dots,du_A\}\quad |A|=k \iff F_t+cF_x=0$$
  and
  \smallskip
  \item[2.] $d\Omega_\theta=0$ if and only if $\Omega_\theta$ is Frobenius integrable and $F_{u_{A-1}}=0,$
  \end{itemize}
  where $\phi:=dx-cdt$ and $|A-1|=k-1.$
  \end{thm}

  \begin{proof}  First of all note that, with $|A|=k,$
  $$dF\in\Sp\{dt,\ dx,\ du,\ du_{xx},\dots, du_A\}$$
  and that
    $$d\theta^1=\frac{1}{c}dF\wedge\phi,\ d\theta^2=cdu_{xx}\wedge\phi,\ d\theta^{|B|}=-du_B\wedge\phi, \ 2<|B|\le k.$$
  Then
\begin{itemize}
  \item[1.] $d\theta^1\wedge\Omega_\theta=0, \ d\theta^2\wedge\Omega_\theta=0,\ \d\theta^{|B|}\wedge\Omega_\theta=0,\ 2<|B|<k$ and
  $$\d\theta^k\wedge\Omega_\theta=-du_A\wedge\phi\wedge du\wedge dF \wedge du_{xx}\wedge\dots\wedge du_{A-1}.$$
  Hence $\Omega_\theta$ is Frobenius integrable if and only if
  $$dF\wedge\phi\equiv 0\ \text{mod}\ \{du,du_{xx},\dots,du_A\}\ \iff  F_t+cF_x=0$$
    \item[2.]
  Now
  $$\quad d\Omega_\theta =d \theta^1\wedge\theta^2\wedge\ldots\wedge \theta^k+\ldots+(-1)^{k-1}\theta^1\wedge\theta^2\wedge\ldots\wedge \theta^{k-1}\wedge d \theta^k,$$
  and every term except the last is trivially zero. The last term is
  \begin{align*}
  &(-1)^{k}du\wedge dF\wedge du_{xx}\wedge\ldots\wedge du_{A-2}\wedge du_A\wedge \phi\\
  =\ &(-1)^{k}du\wedge (F_xdx+F_tdt)\wedge du_{xx}\wedge\ldots\wedge du_{A-2}\wedge du_A\wedge \phi\\
  +\ &(-1)^{k}du\wedge F_{u_{A-1}}du_{A-1}\wedge du_{xx}\wedge\ldots\wedge du_{A-2}\wedge du_A\wedge \phi,
  \end{align*}
  hence the closure result.
\end{itemize}
\end{proof}


\begin{xmpl}
Consider the evolution equation
\begin{equation}
u_t=u_{xxx}\label{Ex1:eq}
\end{equation}
along with the travelling wave ansatz \eqref{Tw:eq}. $D_V^{\ast}$ is spanned by
\begin{align*}\theta^1&:=du+\frac{1}{c}u_{xxx}(dx-cdt),\ \theta^2:=du_{xxx}+cu_{xx}(dx-cdt),\\
\theta^3&:=du_{xx}-u_{xxx}(dx-cdt),
\end{align*}
and $D_V$ is spanned by
\begin{align*}
V_1&:=\vp{t}+u_{xxx}\vp{u}-cu_{xxx}\vp{u_{xx}}+c^2u_{xx}\vp{u_{xxx}},\\
V_2&:=\vp{x}-\frac{1}{c}u_{xxx}\vp{u}+u_{xxx}\vp{u_{xx}}-cu_{xx}\vp{u_{xxx}}.
\end{align*}
$D_V^{\ast}$ is Frobenius integrable with $d\Omega_\theta=0$. A solvable structure for $D_V$ (found with {\tt dimsym} \cite{sherrdimsym}) is $\{X_3,X_2,X_1\}$ where
\begin{align*}
X_1&:=\vp{t}, \ X_2:=\vp{x},\\
X_3&:=u\vp{u}+u_{xx}\vp{u_{xx}}+u_{xxx}\vp{u_{xxx}}.
\end{align*}
It can be immediately seen that $[X_a, V_i]=0,\ [X_1,X_2]=0=[X_1,X_3]$ and $[X_2,X_3]=X_2.$
A simple calculation shows that
$$\mathcal{L}_{X_1}\Omega_\theta=0,\ \mathcal{L}_{X_2}\Omega_\theta=0,\ \mathcal{L}_{X_3}\Omega_\theta=3\Omega_\theta.$$
Proposition \ref{propn2} indicates that $X_1 \righthalfcup X_2 \righthalfcup \Omega_\theta$ is a closed one-form factor of $\Omega_\theta.$ This integrates to the following {\em first integral} of $D_V$ (see the section 4 for a definition):
$$f^1:=cu_{xx}^2+u_{xxx}^2 $$
It is a simple matter to scale $X_3$ to $\tilde{X}_3:=(f^1)^{-\frac{3}{2}}X_3$  so that $\mathcal{L}_{\tilde{X}_3}\Omega_\theta=0$ and $[X_1,\ \tilde{X}_3]=0,\ [X_2,\ \tilde{X}_3]=(f^1)^{-\frac{3}{2}}X_2,$ maintaining the solvable structure $\{\tilde{X}_3,\ X_2,\ X_1\}$. From proposition \ref{propn2} we immediately recover a further independent closed one-form, namely $\tilde{X}_3 \righthalfcup X_1 \righthalfcup \Omega_\theta$  even though $\{X_2,\ \tilde{X}_3,\ X_1\}$  is not itself a solvable structure. As a result we have another first integral of $D_V,$ namely,
$$f^2:=cu+u_{xx}.$$
Using Theorem \ref{sherr result} we obtain a third integral
$$f^3:=x-ct+c^{-\frac{1}{2}}\tan^{-1}\left(c^{-\frac{1}{2}}\frac{u_{xxx}}{u_{xx}}\right)$$
so that the integral manifolds of $D_V$ are the common levels sets of $f^1,\ f^2,\ f^3$ and the solutions of \eqref{Ex1:eq} is the 3 parameter family obtained by eliminating all the derivatives from these codimension 3 level set equations and solving for $u.$

\end{xmpl}

\section{The Korteweg de Vries equation}

The  Korteweg de Vries equation
\begin{equation}
  u_t=-uu_x+u_{xxx}
\end{equation}
produces the canonical example of a nonlinear solitary travelling wave.
Applying the travelling wave ansatz described in the previous section we find that the third order pulled-back  contact system $D_V^{\ast}$ on $J^3(\R^2,\R)$ is generated by
\begin{align*}
  \theta^1&:=du+\frac{F}{c}(dx-cdt),\\
  \theta^2&:=dF+cu_{xx}(dx-cdt),\\
  \theta^3&:=du_{xx}-u_{xxx}(dx-cdt),
\end{align*}
with $F=\frac{c}{c-u}~u_{xxx}.$

\noindent From theorem~\ref{propn 1} $D_V^{\ast}$ is clearly Frobenius integrable, that is
 \[d\theta^a\wedge\Omega_\theta=0.\]
Moreover, and again from theorem~\ref{propn 1}, $d\Omega_\theta=0$ so that locally, \[\Omega_\theta=df^1\wedge df^2\wedge df^3.\]
The corresponding Vessiot distribution $D_V$ is generated by
\begin{align*}
V_{1}&:=\vp{t}+F\vp{u}-cu_{xxx}\vp{u_{xx}}+\frac{c}{(c-u)^2}((c-u)^3u_{xx}-u_{xxx}^2)\vp{u_{xxx}},\\
V_{2}&:=\vp{x}-\frac{F}{c}\vp{u}+u_{xxx}\vp{u_{xx}}-\frac{1}{(c-u)^2}((c-u)^3u_{xx}-u_{xxx}^2)\vp{u_{xxx}}
\end{align*}
The common level sets of $f^\alpha$ project to graphs of travelling wave solutions. There are various ways to find the $f^\alpha,$ for example we could use a solvable structure and proposition \ref{propn2} since $\Omega_\theta$ is closed. However, the KdV problem does not admit a projectable solvable structure. As it happens it is a simple matter to find two of  the $1-$form factors of $\Omega_\theta$ directly and in a way that connects to known results. For the sake of convenience, we re-label the variables  \[y_1:=x-ct,y_2:=u,~y_3:=F,~y_4:=u_{xx}.\]

Now find the closed one-form factors, $dH$, of $\Omega_\theta:$
\begin{equation}\label{dh eq}
dH\wedge\Omega_\theta=0.
\end{equation}
Equation \eqref{dh eq} implies
\[\pd{H}{{y_1}}-\frac{1}{c}{y_3}\pd{H}{{y_2}}-c{y_4}\pd{H}{{y_3}}-\frac{1}{c}(c-u){y_3}\pd{H}{{y_4}}=0.\] So we need integral curves of
\[X:=\vp{y_1}-\frac{1}{c}y_3\vp{y_2}-cy_4\vp{y_3}-\frac{(c-y_2)}{c}y_3\vp{y_4}.\]
This gives the well known ODE
\begin{equation}
  y_2'''+(y_2-c)y_2'=0,
\end{equation}
where $'\equiv \ode{y_1},\quad y_1:=x-ct,\quad y_2:=u.$ That is,
\[ u'''+(u-c)u'=0,\]
This leads directly to closed forms for two of the first integrals:
\begin{align*}
f^1&=u''+u(\frac{1}{2}u-c)=u_{xx}+u(\frac{1}{2}u-c),\\
f^2&={u'}^2+\frac{1}{3}u^2(u-c)-2f^1u = u^2_x+\frac{1}{2}u^2(3c-u)-2uu_{xx}.
\end{align*}
 Restricting to an arbitrary common level set of $f^1,~f^2$ gives
 \begin{align*}
u''+u(\frac{1}{2}u-c)=K,\quad {u'}^2+\frac{1}{3}u^2(u-c)-2f^1u=L.
\end{align*}
Applying the usual boundary conditions at infinity forces $K=0,~L=0$ and  gives  the soliton solution
\[u(x,t)=3c~\text{sech}^2\left(\frac{1}{2}\sqrt{c}(x-ct)+M\right).\]
In this expression $M$ is the value of the third integral, $f^3,$ which is famously indeterminate in the absence of boundary conditions.

\section{Conserved quantities and constants of the motion}

In this section we will demonstrate the construction of constants of the motion for evolution equations. In general these constants of the motion will not apply to all solutions but for those which share one or more first integrals.

\noindent To make things concrete, consider a second order evolution PDE of one dependent variable $u$ and two independent variables $t,~x$ of the form
\begin{equation}
  u_{t}=F(x,t,u,u_x,u_{xx})\label{2evol}
\end{equation}
For the moment we don't assume the travelling wave ansatz \eqref{Tw:eq}.

We give the following definitions:

\begin{defn}
  A {\em first integral} of an integrable distribution $D:=Sp\{X_1,\dots,X_p\}$ on $M^n$ is a smooth, non-constant (local) function on $M^n$ with $X_a(f)=0$ for all $a$.
\end{defn}
 \begin{defn}
A {\em conservation law} of \eqref{2evol},
\begin{equation*}
\pd{T}{t}+\pd{X}{x}=0,
\end{equation*}
involves two suitably integrable functions of $t,x,u,u_x,u_{xx}$, the {\em density} $T$ and the {\em flux} $X$. The resulting equation
\begin{equation*}
\ode{t}\int^a_b~Tdx=0
\end{equation*}
identifies $\int^a_b~Tdx$ as a {\em constant of the motion}. The integrand $T$ is called a {\em conserved density}.
 \end{defn}
Suppose that, as described in section 2, we have a particular and possibly reduced, integrable Vessiot distribution, $D_{V_\text{red}},$ for our evolution equation, spanned by $V_1,~V_2$ on $J^2(\R^2,\R).$ Without loss, they look like,
\begin{align*}
  V_1=\vp{x}+u_x\vp{u}+\ldots~,\quad \quad   V_2=\vp{t}+F\vp{u}+\ldots
\end{align*}
The integral manifolds $S$ of $D_{V_\text{red}}$ are the common level sets of three independent first integrals $f^\alpha$, so that $V_a(f^\alpha)=0$ and $df^1\wedge df^2\wedge df^3\neq 0$ (at least locally on $S$).  We emphasise that these integrals are specific to the particular solution of \eqref{2evol} determined by $D_{V_\text{red}}.$ By construction, $V_1,V_2$ project to the total derivatives $\ode{t},\ode{x}$ along the projected solutions on the total (graph) space $\R^2\times\R$.

Suppose that \[\overline {N}:=\{p\in S: f^\alpha(p)=c^\alpha\}\] is a common level set of the $f^\alpha.$ Then $V_1,~V_2$ are tangent to $\overline {N}$ and let $\gamma_1,~\gamma_2$ be integral curves of $V_1,~V_2$ through some point $p\in \overline {N}.$ Then $\gamma_1,~\gamma_2$ project by $\pi:J^2(\R^2,\R) \to \R^2\times\R$ to curves on the solution surface $N:=\pi(\overline N)$, $u=u(t,x;c^\alpha)$, and
\[\left(\vp{t}+F\vp{u}\right)\Big{|} _{\pi\circ \gamma_1},\quad \left(\vp{x}+u_x\vp{u}\right)\Big{|} _{\pi\circ \gamma_2}\] are tangent to these projected curves on $N.$ See Figure \ref{figure1}.

\usetikzlibrary{decorations.pathreplacing,decorations.markings}
\begin{figure}

\begin{center}

\begin{tikzpicture}[scale=0.4]

\def\a{6}

\coordinate (A1) at (4.75, 3.5+\a, 3);
    \coordinate (A2) at (4, 4.5+\a, -3);
    \coordinate (A3) at (-4, 5.5+\a, -4);
   \coordinate (A4) at (-3.75, 5.5+\a, 5);
%

  \draw[-, bend  left=18] (A1) to (A2) node [midway, right, xshift=5cm,yshift=10cm]{$\overline{N}\subset S$};
    \draw[-, bend right=8] (A2) to (A3);
    \draw[-, bend right=18] (A3) to (A4);
     \draw[-, bend left=8] (A4) to (A1);

\draw[-, bend  left=8] (-1.4,3.8+\a,0.5)node [below right,xshift=0.2cm,yshift=0.3cm] {\Tiny $\gamma_1$}  to (4,6+\a,1);
\draw[-, bend  left=7] (-2,6.2+\a,1)  to (4,4+\a,1)node [below,xshift=-0.2cm,yshift=0.1cm] {\Tiny $\gamma_2$};

     \draw[fill=black] (1.59, 5.37+\a, 1.5) circle(0.07)node [above] {\Tiny $p$};
   \draw[-<] (4, 4+\a, 1) --(3.5, 4.3+\a, 1.1);

     \draw[<-] (-1.2, 3.8+\a, -0.2) --(-0.9, 3.9+\a, -0.42);

 \draw[->, ultra thick] (1.59, 5.37+\a, 1.5)  to (0.7,5.15+\a,1.7)node [below,right,xshift=-0.1cm,yshift=-0.2cm] {\Tiny $V_1$};
   \draw[->, ultra thick] (1.59, 5.37+\a, 1.5)  to (-1.8,0.57+\a,-10)node [below,right,xshift=-0.3cm,yshift=-0.17cm] {\Tiny $V_2$};

        \draw[->](0.2,9,0.3)--(0.1,7.5,0.2)node[midway, left]{$\pi$};

\coordinate (A1) at (4.75, 3.5, 3);
    \coordinate (A2) at (4, 4.5, -3);
    \coordinate (A3) at (-4, 5.5, -4);
   \coordinate (A4) at (-3.75, 5.5, 5);
%

  \draw[-, bend  left=18] (A1) to (A2)node [midway, right, xshift=5cm,yshift=4cm]{ $N\subset  \R^2\times\R$};
    \draw[-, bend right=8] (A2) to (A3);
    \draw[-, bend right=18] (A3) to (A4);
     \draw[-, bend left=8] (A4) to (A1);


     \draw[-, bend  left=8] (-1.4,3.8,0.5)node [below right,xshift=0.2cm,yshift=0.3cm] {\Tiny {$\pi\circ\gamma_1$}}  to (4,6,1);
      \draw[-, bend  left=7] (-2,6.2,1)  to (4,4,1)node [below,xshift=-0.5cm,yshift=0.1cm] {\Tiny {$\pi\circ\gamma_2$}};
     \draw[fill=black] (1.53, 5.4, 1.53) circle(0.06)node [above,xshift=-0.28cm,yshift=0.1cm] {\Tiny {$\pi(p)$}};
\draw[<-] (-1.2, 3.8, -0.2) --(-0.9, 3.9, -0.42);
\draw[-<] (4, 4, 1) --(3.5, 4.3, 1.1);

   \draw[->] (-4,-1,-2) -- (-2,-1.7,-2) node[anchor=north east]{$x$};
\draw[-] (-4,-1,-2) -- (-5,0.2,-4.6) node[anchor=north west]{$u$};
\draw[->] (-4,-1,-2) -- (-2.4,1.4,6.8) node[anchor=south]{$t$};

\draw[->] (1.5,-2,1.5) -- (3.5,-2.1,2.9) node[right]{$\frac {\partial}{\partial x}$};

\draw[dashed,-] (1.5,-2,1.5) -- (0.8,5.4,-0.4) node[left]{};
\draw[->] (1.5,-2,1.5) -- (-0.5,-3.5,-0.1) node[left]{$\frac {\partial}{\partial t}$};
\draw[dashed, -](1.5,-0.5,1.5)--(0.8,5.5,-0.4);
\draw[dashed, -](1.5,6.5,1.46)--(0.8,10.7,-0.4);
\draw[fill=black](1.5,-2,1.5) circle(0.03);

\end{tikzpicture}
\end{center}
\caption{} \label{figure1}
\end{figure}

Now consider the meaning of  \[I:=\int_a^b~G~dx \quad \text{with}\quad G=G(x,t,u,\ldots,u_{xx},\dots).\]
For the solution corresponding to $D_{V_\text{red}}$
\[\int_a^b~G~dx=\int_a^b~G\circ\gamma_2~dx,\]
that is, the integral of $G$ on $S$ along an integral curve $\gamma_2: [a,b] \to \bar N$ of $V_2$ from
\begin{align*}
  \gamma_2(a)&:=(t,a,u(t,a),\ldots,u_{xx}(t,a,u(t,a)))\\
  \text{to} &\quad (t,b,u(t,b),\ldots,u_{xx}(t,b,u(t,b))=:\gamma_2(b)
\end{align*}
Thus the functional $I$ is an integral of the restriction of $G$ to a solution surface and of course all the integral curves of $V_2$ project to curves on solution surfaces.

Moreover, the total time derivative of $I$ has the meaning $V_1\left(\int_{a}^{b}~G\circ \gamma_2~dx\right),$ because total time differentiation is just the projection of the directional derivative in the direction of $V_1.$

Now $[V_1,V_2]=0$ by construction, so
\[V_1\Big{(}\int_{a}^{b}~G\circ \gamma_2~dx\Big{)}=\int_{a}^{b}~V_1(G)\circ \gamma_2~dx.\]
Clearly, if $G$ is a composite of $f^\alpha$ then $V_1(G)=0.$ Further, $G$ then takes a single constant value, $l$ say, on the integral manifold of $D_{V_\text{red}}^{\ast}$ on which $\gamma_2$ lies, so
  \[\int_{a}^{b}~G\circ \gamma_2~dx=l(b-a),\]
and
 \[V_1\Big{(}\int_{a}^{b}~G\circ \gamma_2~dx\Big{)}=0.\]
Typical conservation laws look like
\[\ode{t}\int_a^b~G(x,y,u,\ldots,u_{xxx},\ldots)~dx=0,\]
and so any composite $G$ of the $f$'s  produces a ``constant of the motion'' for the solution determined by these $f$'s.

As an example consider all regular solutions of the Korteweg de Vries equation satisfying the travelling wave ansatz \eqref{Tw:eq}. They all have first integral $f^1:=u_{xx}+u(\frac{1}{2}u-c)$ and so
\[\int_{a}^{b}~u_{xx}+u(\frac{1}{2}u-c)~dx\]
is a constant of the motion for some suitable interval $[a,b]$.

Returning to the general situation for $G:S\rightarrow \R$, where $G$ is not a composite of first integrals, but where we continue to consider the solution determined by our particular $D_{V_\text{red}},$
\[\ode{t}\int_a^b~G~dx=V_1\Big{(}\int_{a}^{b}~G\circ \gamma_2~dx\Big{)}=\int_{a}^{b}~V_1(G)\circ\gamma_2~dx.\]

The integral curve $\gamma_2:[a,b]\rightarrow \bar N$ represents an element of a family of curves obtained by varying the value of the $t$ co-ordinate of the point $\gamma_2(a)\in S.$ If $g_t:S\rightarrow S$ is the one-parameter group generated by $V_1$ then these curves are the images $g_t\circ\gamma_2$ of $\gamma_2.$
Because $[V_1,V_2]=0,$ this family of curves lies on an integral manifold of $D_{V_\text{red}}^{\ast},$ that is a lifted solution of the PDE and a common level set of the $f^\alpha.$ See Figure \ref{figure2}.

\begin{figure}

\begin{center}

\begin{tikzpicture}[scale=0.7]

\def\a{0.5}
\def\b{1}
\coordinate (A1) at (5.25, 1.5, 3);
    \coordinate (A2) at (5.5, 1.5, -2);
    \coordinate (A3) at (-1, 3.5, 0);
   \coordinate (A4) at (-.75, 1.5, 1.5);
%

  \draw[-, bend  left=8] (A1) to (A2);
    \draw[-, bend right=22] (A2) to (A3);
    \draw[-, bend right=8] (A3) to (A4);
     \draw[-, bend left] (A4) to (A1);

%

\draw[->] (-2,-2,-2) -- (4.5,-2,-2) node[anchor=north east]{\Tiny $x$};
\draw[-] (-2,-2,-2) -- (-2,0,-2) node[anchor=north west]{\Tiny $u$};
\draw[->] (-2,-2,-2) -- (-2,-2,2) node[anchor=south]{\Tiny$t$};

\draw[->] (1.99,-2,-2) -- (2,-2,-2);

\draw [-](0.5, 2.75, 1.5)node [above, xshift=1cm, yshift=0.2cm] {} .. controls (2.5,3.5, 1.5)and (2.75,2, 1.5) .. (4.5,2.5, 1.5)node [right, xshift=1cm,yshift=0.3cm] {\Tiny $\bar{N}$} ;
\draw (0.5, 2.75, 1.5) circle (0.03);
\draw (4.5,2.5, 1.5) circle (0.03);

\draw [thick, ->](1.5, 2.96, 1.5)--(1.52, 2.96, 1.5);

\draw [-](0.5+\a, 2.75+\b, 1.5)node [above, xshift=1cm, yshift=0.03cm] {\tiny $\gamma_2$} .. controls (2.5+\a,3.5+\b, 1.5) and (2.75+\a, 2+\b, 1.5) .. (4.5+\a, 2.5+\b, 1.5);
\draw (0.5+\a, 2.75+ \b, 1.5) circle (0.03);
\draw (4.5+\a,2.5+ \b, 1.5) circle (0.03);

\draw [thick, ->](2, 3.96, 1.5)--(2.02, 3.96, 1.5);

\draw[dashed](0.5+\a, 2.75+ \b, 1.5)--(0.5+\a, -0.65, 1.5)[fill=black]circle (0.03)node[above, right, yshift=0.15cm]{\Tiny $a$};
\draw[dashed](4.5+\a,2.5+ \b, 1.5)--(4.5+\a,-0.65, 1.5)[fill=black]circle (0.03)node[above, right, yshift=0.15cm]{\Tiny $b$};

\draw[dashed](0.5, 2.75, 1.5)--(0.5, -1.25, 1.5);
\draw[dashed](4.5,2.5, 1.5)--(4.5,-1.25, 1.5);

\draw [-](0.5, -1.25, 1.5)[fill=black]circle (0.03)--(4.5,-1.25, 1.5)[fill=black]circle (0.03);
\draw [->](1.5, -1.25, 1.5)--(1.56, -1.25, 1.5);

\draw[->](0.35+\a, 2.5+ \b, 1.5)--(0.6, 3, 1.5)node[midway, left]{\Tiny$g_t$};

\draw[->](4.35+\a,2.25+ \b, 1.5)--(4.6,2.75, 1.5) node[midway, left]{};

\draw [dashed](0.5, -1.25, 1.5)--(0.5, -1.25, 4);
\draw [dashed](4.5,-1.25, 1.5)-- (4.5,-1.25, 4);

\draw [->](5,-0.8, 1.65)--(5,-0.8, 2.25)node[midway, right]{\Tiny$t$};
\end{tikzpicture}
\end{center}
\caption{}\label{figure2}
\end{figure}
\bigskip
In the situation, where \[\ode{t}\int_a^b~G~dx=0,\] we have
\[V_1\Big{(}\int_{a}^{b}~G\circ\gamma_2~dx\Big{)}=0=\int_{a}^{b}~V_1(G)\circ\gamma_2~dx.\]

The interpretation is that the integral is independent of the integral curve $g_t\circ\gamma_2$ along which the integral is performed. This does not imply that $V_1(G)=0.$  In other words, conserved densities are not necessarily first integrals but first integrals are always conserved densities.

A simple example of this is the well-known conserved density $T_3:=u^3+\frac{1}{2}u_x^2$ of the Korteweg de Vries equation. For all travelling wave solutions
\[\ode{t}\int_{-a}^a~T_3~dx=\int_{-a}^a~V_1(T_3)~dx=-c\int_{-a}^a~\frac{du^3}{dx}~dx=0,\]
since $V_1(T_3)=V_1(u^3)=-cV_2(u^3)$, while $T_3$ is clearly not a first integral.

\section*{Acknowledgements}
Naghmana Tehseen acknowledges the support of a La Trobe University  postgraduate research award. She would also like to thank Peter Vassiliou for helpful discussions and Mumtaz Hussain for helping with the diagrams.


\providecommand{\bysame}{\leavevmode\hbox to3em{\hrulefill}\thinspace}
\providecommand{\MR}{\relax\ifhmode\unskip\space\fi MR }
\providecommand{\MRhref}[2]{%
  \href{http://www.ams.org/mathscinet-getitem?mr=#1}{#2}
}
\providecommand{\href}[2]{#2}

\end{document}